\documentclass[12pt]{article}
\usepackage{graphicx}
\usepackage{amsmath,amssymb,amsthm,amsfonts}

\newtheorem{thm}{Theorem}[section]
\newtheorem{cor}[thm]{Corollary}
\newtheorem{lem}[thm]{Lemma}

\newcommand{\lam}{\lambda}

\newcommand{\R}{{\mathbb{R}}}
\newcommand{\Z}{{\mathbb{Z}}}

\newcommand{\1}{\partial}
\newcommand{\2}{\overline}
\newcommand{\3}{\varepsilon}
\newcommand{\4}{\widetilde}

\begin{document}
\title{Quenching behaviour of a nonlocal parabolic MEMS equation}
\author{Kin Ming Hui\\
Institute of Mathematics, Academia Sinica,\\
Nankang, Taipei, 11529, Taiwan, R. O. C.}
\date{March 16, 2010}
\smallbreak \maketitle
\begin{abstract}
Let $\Omega\subset\mathbb{R}^n$ be a $C^2$ bounded domain and 
$\chi>0$, $\lambda>0$, be constants. We obtain upper bounds 
for the quenching time of the solutions of the nonlocal parabolic 
MEMS equation $u_t=\Delta u
+\lam/(1-u)^2(1+\chi\int_{\Omega}1/(1-u)\,dx)^2$ 
in $\Omega\times (0,\infty)$, $u=0$ on $\1\Omega\times (0,\infty)$, 
$u(x,0)=u_0$ in $\Omega$, when $\lambda$ is large. We will prove 
the compactness of the quenching set under a mild condition on 
$u_0$. When $\Omega=B_R=\{x\in\R^n:|x|<R\}$ and $u_0$ is radially 
symmetric and monotone decreasing in $0\le r\le R$, we prove that 
the point $x=0$ is the only possible quenching set. When 
$\Omega=B_R$ and $u_0$ is radially 
symmetric which also satisfies some strict concavity assumption, 
we prove that for any $\beta\in (2,3)$ the solution satisfies 
$1-u(x,t)\ge C|x|^{\frac{2}{\beta}}$ for some constant $C>0$ and
that the solution $u$ quenches in a finite time for any sufficiently 
large $\lambda>0$. We also obtain the quenching time
estimate in this case.  
\end{abstract}

\vskip 0.2truein

Key words: parabolic nonlocal MEMS equation, quenching time estimates, 
compactness, quenching set

Mathematics Subject Classification: Primary 35B40 Secondary 35B05, 35K50, 
35K20
\vskip 0.2truein
\setcounter{equation}{0}
\setcounter{section}{0}

\setcounter{equation}{0}
\setcounter{thm}{0}

Micro-electromechanical systems (MEMS) devices are key components of many 
electronic devices including accelerometers for airbag deployment in cars, 
inkjet printer heads, and the device for the protection of hard disk, etc. 
It is therefore interesting to understand the mathematical modelling of
the MEMS devices and study the various properties of such models. 
Interested readers can read the book, Modeling MEMS and NEMS \cite{PB}, 
by J.A.~Pelesko and D.H.~Berstein for the mathematical modeling and various 
applications of MEMS devices.

One model of MEMS \cite{GPW}, \cite{LY}, \cite{PB}, consists of a fixed 
membrane and a deformable  membrane of the same shape which 
is coated with a thin dielectric material and held fixed 
at the boundary. When no voltage is applied to the membranes, the 
two membranes are parallel to each other with one at a fixed small distance
on top of the other. When a voltage is applied to the membranes, the 
deformable membrane will deflect towards the fixed membrane. Recently 
there are a lot of study on the equations arising from such model of MEMS by 
N.~Ghoussoub, Y.~Guo, Z.~Pan and M.J.~Ward \cite{GG1}, \cite{GG2},
\cite{GG3}, \cite{GG4}, \cite{G}, \cite{GPW}, J.S.~Guo, B.~Hu and C.J.~Wang \cite{GHW},
K.M.~Hui \cite{H1}, \cite{H2}, 
\cite{H3}, N.I.~Kavallaris, T.~Miyasita and T.~Suzuki \cite{KMS}, 
F.~Lin and Y.~Yang \cite{LY}, L.~Ma, Z.~Guo and J.C.~Wei \cite{GW1},
\cite{GW2}, \cite{MW}, G. Flores, 
G.A.~Mercado, J.A.~Pelesko and A.A.~Triolo  \cite{FMP}, \cite{P}, 
\cite{PT} etc. 

As observed by N.~Ghoussoub, Y.~Guo, J.A.~Pelesko and others \cite{GG1},
\cite{PB}, the deflection of the deformable membrane from its equilibrium 
position is modeled by the following parabolic equation,
\begin{equation*}\left\{\begin{aligned}
\frac{\1 u}{\1 t}=&\Delta u
+\frac{\lambda f(x)}{(1-u)^2} 
\quad\quad\mbox{in }\Omega\times (0,T)\\
u=&0\qquad\qquad\qquad\qquad
\mbox{on }\partial\Omega\times (0,T)\\
u(x,0)=&u_0 \qquad\qquad\qquad\quad\mbox{ in }\Omega
\end{aligned}\right.\tag {$P_{\lambda}'$}
\end{equation*}
where $\lambda\ge 0$ is a constant and $f\not\equiv 0$ is a nonnegative 
function on $\2{\Omega}$ which depends
on the dielectric constant of the coating on the membrane. When the 
voltage  between the membranes are due to circuit series capacitance, 
the deflection of the deformable membrane from its equilibrium position 
is modeled by the following nonlocal parabolic equation \cite{GG5},
\begin{equation*}\left\{\begin{aligned}
\frac{\1 u}{\1 t}=&\Delta u
+\frac{\lambda}{(1-u)^2(1+\chi\int_\Omega\frac{dy}{1-u(y,t)})^2} 
\quad\quad\mbox{in }\Omega\times (0,T)\\
u=&0\qquad\qquad\qquad \qquad\qquad\qquad\quad\quad\quad\quad
\mbox{ on }\partial\Omega\times (0,T)\\
u(x,0)=&u_0 \qquad\quad\qquad\qquad\qquad\qquad\qquad\quad\quad\,\,
\mbox{ in }\Omega
\end{aligned}\right.\tag {$P_{\lambda}$}
\end{equation*}
where $\lambda\ge 0$ and $\chi>0$ are a constant. In \cite{LY} F.H.~Lin 
and Y.~Yang by using variational argument derived the following 
nonlocal MEMS equation
\begin{equation*}\displaystyle
\left\{\begin{aligned}
-\Delta v=&\frac{\lam}{(1-v)^2(1+\chi\int_{\Omega}\frac{dx}{1-v})^2}
\quad\mbox{ in }\Omega\\
v=&0\quad\qquad\qquad\qquad\qquad\quad\,\,\mbox{on }\1\Omega
\end{aligned}\right.
\end{equation*} 
for modeling the stationary deflection between the two parallel plates 
of an electrostatic MEMS device with circuit series capacitance. 

Note that $\lambda$ is proportional to the square of the voltage. When 
the voltage is very large, the two membranes will touch each other or 
quench and we would expect the solutions of ($P_{\lambda}'$) and 
($P_{\lambda}$) cease to exist after a finite time. One would like
to get estimates for the touchdown time of the solutions of ($P_{\lambda}$) 
and ($P_{\lambda}'$). We refer the readers to the papers \cite{GG2},
\cite{GG3}, \cite{GG4}, by N.~Ghoussoub and Y.~Guo for various touchdown 
time estimates for the solutions of ($P_{\lambda}'$).  

In this paper we will obtain touchdown time estimates for the solutions 
of ($P_{\lambda}$). We prove that the quenching set of ($P_{\lambda}$)
is compact under a mild condition on the initial value of the solution. 
When $\Omega=B_R=\{x\in\R^n:|x|<R\}$ and $u_0$
is radially symmetric and monotone decreasing in $0\le r\le R$, 
we prove that the solution $u$ of ($P_{\lambda}$) satisfies 
\begin{equation}
1-u(x,t)\ge C|x|^2\quad\mbox{ in }B_R\times [t_0,T) 
\end{equation}
for some constants $t_0\in (0,T)$, $C>0$, 
and the point $x=0$ is the only possible 
quenching set. When $\Omega=B_R$ and $u_0$ is radially symmetric 
which also satisfies some strict concavity assumption, 
we prove that for any $\beta\in (2,3)$ the solution 
$u$ of ($P_{\lambda}$) satisfies 
\begin{equation}
1-u(x,t)\ge C|x|^{\frac{2}{\beta}}\quad\mbox{in }Q_R^T 
\end{equation}
for some constant $C>0$ and that the solution $u$ quenches in a 
finite time for any sufficiently large $\lambda>0$. We also obtain 
the quenching time estimate in this case.  

We will assume that $\Omega\subset\R^n$ is a bounded $C^2$ domain 
and $\lambda>0$, $\chi\ge 0$, for the rest of the paper. We start 
with some definitions. For any $\delta>0$, 
$R>0$, $T>0$, let $\Omega_{\delta}=\{x\in\Omega:\mbox{dist}(x,\1\Omega)
<\delta\}$ and $Q_R^T=B_R\times (0,T)$. 
Observe that when $\Omega$ is a bounded
convex domain then there exists a constant $\delta_0>0$ such that for 
any $x\in\Omega_{\delta_0}$ there exists a unique point $y\in\1\Omega$ 
such that the line seqment $\2{xy}$ is perpendicular to $\1\Omega$ at 
$y$.   

For any constants $\chi\ge 0$, $\lambda>0$, 
and
\begin{equation}
u_0\in L^1(\Omega)\text{ with }u_0\le a\text{ a.e. in }\Omega
\end{equation}
for some constant $0<a<1$ we say that $u$ is a solution 
(subsolution, supersolution respectively) of ($P_{\lambda})$
in $\Omega\times (0,T)$ if $u\in C^{2,1}(\Omega\times (0,T))
\cap C(\2{\Omega}\times (0,T))$, $u<1$, satisfies 
$$
\frac{\1 u}{\1 t}=\Delta u
+\frac{\lambda}{(1-u)^2(1+\chi\int_\Omega\frac{dy}{1-u(y,t)})^2} 
\quad\quad\mbox{in }\Omega\times (0,T)
$$
($\le$, $\ge$ respectively) in the classical sense with $u(x,t)=0$ 
($\le$, $\ge$ respectively) on $\1\Omega\times (0,T)$, 
$$
\sup_{\2{\Omega}\times [0,T']}u(x,t)<1\quad\forall 0<T'<T,
$$ 
and 
\begin{equation*}
\|u(\cdot,t)-u_0\|_{L^1(\Omega)}\to 0\quad\text{ as }t\to 0.
\end{equation*}

For any solution $u$ of ($P_{\lambda}$) we define the 
quenching time or touchdown time $T_{\lambda}>0$ as the time 
which satisfies 
$$\left\{\aligned
\sup_{\Omega}u(x,t)&<1\quad\forall 0<t<T_{\lambda}\\
\lim_{t\nearrow T_{\lambda}}\sup_{\Omega}u(x,t)&=1.
\endaligned\right.
$$
We say that $u$ has a finite quenching time if $T_{\lambda}<\infty$ and
we say that $u$ quenches at time infinity if $T_{\lambda}=\infty$. 
For any solution $u$ of ($P_{\lambda}$) we let the quenching set of $u$
to be the set of points $x\in\Omega$ such that there exists a sequence 
$(x_k,t_k)\in\Omega\times (0,T_{\lambda})$ such that $x_k\to x$
and $u(x_k,t_k)\to 1$ as $k\to\infty$.

We first recall some results of \cite{H2} and \cite{H3}.

\begin{thm}(Theorem 2.1 and Theorem 2.2 of \cite{H3})
Let 
\begin{equation}
-b\le u_0\le a\quad\mbox{ in }\Omega
\end{equation} 
for some constants $0<a<1$ and $b\ge 0$. Then for any $\lambda>0$ and 
$\chi>0$ there 
exists $T>0$ such that ($P_{\lambda}$) has a unique solution $-b\le u
<1$ in $\Omega\times (0,T)$ which satisfies
\begin{equation}
u(x,t)=\int_{\Omega}G(x,y,t)u_0(y)\,dy+\lambda\int_0^t\int_{\Omega}
\frac{G(x,y,t-s)}{(1-u(y,s))^2(1+\chi\int_{\Omega}\frac{dz}{1-u(z,s)})^2}
\,dy\,ds
\end{equation}
for all $(x,t)\in\Omega\times (0,T)$ where $G(x,y,t)$ is the Dirichlet
Green function for the heat equation in $\Omega\times (0,T)$. 
\end{thm}

\begin{lem}(cf. Theorem 2.1 of \cite{H2})
Let $u_{0,1}, u_{0,2}\in L^1(\Omega)$ be such that $0\le u_{0,1}
\le u_{0,2}\le a$ in $\Omega$ for some constant $0<a<1$. Let 
$\lambda>0$ and
$0\le f\in C(\2{\Omega}\times (0,T))\cap L^{\infty}(\2{\Omega}\times 
(0,T))$. Suppose $u_1$, $u_2$, are nonnegative subsolution and 
supersolution of ($P_{\lambda}'$) in $\Omega\times (0,T)$ with 
initial value $u_0=u_{0,1}, u_{0,2}$, respectively. Then 
$u_1\le u_2$ in $\2{\Omega}\times (0,T)$.
\end{lem}

By an argument similar to the proof of Proposition 2.1 of \cite{GG3}
we have the following theorem.

\begin{lem}
Let $\Omega\subset\R^n$ be a bounded convex domain. Let $u_0$ satisfy
(4) for some constants $0<a<1$ and $b\ge 0$.
Suppose there exists a constant $\delta_0>0$ such that 
\begin{equation}
\left.u_0\right|_{\2{\Omega}_{\delta_0}}\in C^1(\2{\Omega}_{\delta_0})
\quad\mbox{ and }\quad\frac{\1 u_0}{\1\nu}<0\mbox{ on }\1\Omega 
\end{equation}
where $\1/\1\nu$ is differentiation with respect to the unit outward
normal $\nu$ on $\1\Omega$ and $u$ is the unique solution of ($P_{\lambda}$) 
in $\Omega\times (0,T)$ given by Theorem 1. Then there exists a constant
$\delta_1\in (0,\delta_0)$ such that for any $x\in\2{\Omega}_{\delta_1}$ if 
$y$ is the unique point on $\1\Omega$ such that the line segment 
$\2{xy}$ is perpendicular to $\1\Omega$ at $y$, then 
\begin{equation*}
\frac{\1}{\1\vec{n}}u(z,t)<0\quad\forall z\in\2{xy},0<t<T 
\end{equation*}
where $\1/\1\vec{n}$ is differentiation with respect to the unit vector
$\vec{n}$ along the direction $\vec{xy}$.
\end{lem}

By Lemma 3 and an argument similar to the proof of Proposition 1.3
of \cite{H3} and Propsotion 2.1 of \cite{G} we have the following 
two corollaries.

\begin{cor}
Let $\Omega\subset\R^n$ be a bounded convex domain. Let $u_0$ satisfy
(4) and (6) for some constants $0<a<1$, $b\ge 0$,  and $\delta_0>0$. 
Suppose $u$ is the unique solution of ($P_{\lambda}$) in $\Omega\times 
(0,T)$ given by Theorem 1. Let $\delta_1$ be as in Theorem 3. Then 
there exist $a_1>0$ and $\delta\in (0,\delta_1)$ such that for any 
$y\in\Omega_{\delta}$ there exists a fixed-sized cone $\Gamma(y)\subset
\Omega_{2\delta}$ with vertex at $y$ such that $|\Gamma(y)\setminus 
\Omega_{\delta}|\ge a_1$ and $u(z,t)\ge u(y,t)$ for any $z\in \Gamma(y)$
and $0<t<T$. Moreover 
\begin{equation}
\int_{\Omega_{\delta}}\frac{dx}{(1-u(x,t))^2}
\le\frac{|\Omega|}{a_1}\int_{\Omega\setminus\Omega_{\delta}}
\frac{dx}{(1-u(x,t))^2}\quad\forall 0<t<T.
\end{equation}
\end{cor}

\begin{cor}
Let $\Omega\subset\R^n$ be a bounded convex domain. Let $u_0$ satisfy
(4) and (6) for some constants $0<a<1$, $b\ge 0$,  and $\delta_0>0$. 
Suppose $u$ is the unique solution of ($P_{\lambda}$) in $\Omega\times 
(0,T)$ given by Theorem 1 such that $u$ touchdown at time $T$. Then
the set of touchdown points for $u$ is a compact subset of $\Omega$.
\end{cor}

By (7) and an argument similar to the proof of Theorem 4.3 of \cite{H3}
we have the following theorem. 

\begin{thm}
Let $\Omega\subset\R^n$ be a bounded convex domain. Let $u_0$ satisfy
(4) and (6) for some constants $0<a<1$, $b=0$, and $\delta_0>0$. 
Let $\chi>0$. Then there exists a constant $C_1>0$ and such 
that for any $\lambda>\lambda_0=C_1\mu_1$ and any solution $u$ of 
($P_{\lambda})$, $u$ quenches in a finite time 
\begin{equation*}
T_{\lambda}\le\frac{C_1}{\lambda-\lambda_0}.
\end{equation*}
\end{thm}

\begin{lem}
Let $\lambda>0$, $\chi>0$, and $\Omega\subset\R^n$ be a bounded
domain such that $B_R\subset\Omega$ for some constant $R>0$. Let 
$0\le u_0\le a$ in $\Omega$ for some constant $0<a<1$ and let 
$0\le u<1$ be a solution of $(P_{\lambda})$ in $\Omega\times 
(0,T)$ with
\begin{equation*}
\sup_{0\le t<T}\int_{\Omega}\frac{dz}{1-u(z,t)}<\infty.
\end{equation*} 
Let 
\begin{equation}
A(t)=\biggl(1+\chi\int_{\Omega}\frac{dz}{1-u(z,t)}\biggr)^{-2}
\end{equation} 
and $0<\delta_1\le\inf_{0\le t<T}A(t)$.
Suppose $\lambda>2n/\delta_1R^2$. Then $T$ satisfies
\begin{equation}
T\le\frac{1}{\lambda\delta_1}\biggl(1-\frac{2n}{\lambda\delta_1R^2}
\biggr)^{-1}.
\end{equation}
\end{lem}
\begin{proof}
Let $v=1-u$ and $v_0=1-u_0$. Then $v$ satisfies 
\begin{equation}
\left\{\begin{aligned}
v_t&=\Delta v-\lambda A(t)v^{-2}\quad\mbox{ in }\Omega\times (0,T)\\
v(x,t)&=1\qquad\qquad\qquad\quad\mbox{ on }\1\Omega\times (0,T)\\
v(x,0)&=v_0\qquad\qquad\qquad\,\,\,\mbox{ in }\Omega.
\end{aligned}\right.
\end{equation} 
Let
$$
\psi (x,t)=1-\lambda\delta_1c_0t\biggl (1-\frac{|x|^2}{R^2}\biggr )
$$
where
\begin{equation*}
c_0=1-\frac{2n}{\lambda\delta_1R^2}.
\end{equation*} 
Then
\begin{equation}
\frac{2n}{R^2}\cdot\frac{1}{\lambda\delta_1 c_0}=c_0^{-1}-1.
\end{equation} 
Hence by (11),
\begin{equation*}
\psi_t-\Delta\psi+\lambda\delta_1\psi^{-2}
\ge \lambda\delta_1c_0[-1-(2nt/R^2)+c_0^{-1}]\ge 0
\quad\mbox{ in }Q_R^{T_1}
\end{equation*} 
where $T_1=\min (T,1/(\lambda\delta_1 c_0))$. Hence $\psi$ is a 
supersolution of
\begin{equation}
\left\{\begin{aligned}
\psi_t&=\Delta\psi-\lambda\delta_1\psi^{-2}\quad\mbox{ in }Q_R^{T_1}\\
\psi&=1\qquad\qquad\qquad\mbox{ on }\1 B_R\times [0,T_1)\cap
\2{B}_R\times\{0\}.
\end{aligned}\right.
\end{equation} 
Suppose $T>1/(\lambda\delta_1 c_0)$. Since $v$ is a subsolution 
of (12), by Lemma 2,
\begin{align*}
&0<v\le\psi\qquad\mbox{ in }\2{Q_R^{T_1}}\\
\Rightarrow\quad&0<v (0,1/(\lambda\delta_1 c_0))\le\psi 
(0,1/(\lambda\delta_1 c_0))=0.
\end{align*}
Contradiction arises and (9) follows.
\end{proof}

\begin{thm}
Let $0\le u_0\le a$ be a radially symmetric function in $B_R\subset\R^n$ 
for some constant $0<a<1$ which is monotone decreasing in $0\le r\le R$. 
Let $T>0$. Suppose $u$ is the unique solution of ($P_{\lambda}$) in 
$Q_R^T$ given by Theorem 1 and $u$ touches down at time $T$. Then 
there exist constants $t_0\in (0,T)$ and $C>0$ depending on 
$\lambda$ and $\chi$ such that (1) holds.
Hence $x=0$ is the only quenching point of $u$ at time $T$. Moreover 
for $n\ge 3$ we have 
\begin{equation}
\sup_{0\le t<T}\int_{B_R}\frac{dz}{1-u(z,t)}<\infty.
\end{equation}
\end{thm}
\begin{proof}
We first observe that by an argument similar to the proof of Theorem 4.3 of 
\cite{H3} and Theorem 1.5 of \cite{Hs} $1>u(x,t)=u(|x|,t)\ge 0$ is 
radially symmetric in $Q_R^T$ with 
\begin{equation}
u_r(0,t)=0\quad\mbox{ and }u_r(r,t)<0\quad\forall 0<r\le R, 0<t<T.
\end{equation}  
Let
$$
v_1(x,t)=\int_{B_R}G(x,y,t)u_0(y)\,dy
$$
and
$$
F(x,t,s)=\lambda\int_{B_R}G(x,y,t-s)A(s)(1-u(y,s))^{-2}\,dy
$$
where $A(t)$ is given by (8) with $\Omega=B_R$. Since 
$0\le v_1\le a$ satisfies the heat equation 
in $Q_R^T$ with $v_1\equiv 0$ on $\1 B_R\times (0,T)$ and $u_0$ 
is radially symmetric and monotone decreasing for $0\le r\le R$,
by the maximum principle and an argument similar to the proof of 
Theorem 1.5 of \cite{Hs} $v_1$ is radially symmetric and 
$v_{1,r}(r,t)\le 0$ for any $0<r\le R$ and $0<t<T$.   

Similarly for any $0<s<T$, $F(\cdot,\cdot,s)$ is a radially symmetric 
solution of the heat equation in $B_R\times (s,T)$ with 
$F(x,t,s)\equiv 0$ for any $x\in \1 B_R$, $0<s<t<T$, and
\begin{equation}
\frac{\1 F}{\1 r}(r,t,s)<0\quad\forall 0<r=|x|\le R, 0<s<t<T.
\end{equation}
By Theorem 1 $u$ satisfies (5). By (5) and (15),
\begin{equation}
u_r(r,t)=v_{1,r}(r,t)+\int_0^tF_r(r,t,s)\,ds\le\int_0^{\frac{T}{3}}
F_r(r,t,s)\,ds\quad\forall 0\le r\le R,T/3\le t<T.
\end{equation}
Let
\begin{equation}
v_2(x,t)=\int_0^{\frac{T}{3}}F(x,t,s)\,ds
\quad\forall |x|\le R,T/3\le t<T.
\end{equation}
Then $v_2$ is radially symmetric and satisfies
\begin{equation}
\left\{\begin{aligned}
v_{2,t}&=\Delta v_2=\frac{1}{r^{n-1}}(r^{n-1}v_{2,r})_r
\quad\mbox{ in }B_R\times (T/3,T)\\
v_2&=0\qquad\qquad\qquad\qquad\qquad\mbox{ on }\1 B_R\times (T/3,T)\\
v_2&>0\qquad\qquad\qquad\qquad\qquad\mbox{ in }B_R\times\{T/3\}
\end{aligned}\right.
\end{equation}
By (15) for each $T/3<t<T$ $v_{2,r}(r,t)$ is monotone decreasing in 
$0\le r\le R$. Hence for each $T/3<t<T$ $v_2(x,t)$ will attain its 
maximum at $x=0$. Differentiating (18) with respect to $r$ and 
letting $q=r^{n-1}v_{2,r}$,
\begin{equation}
q_t=q_{rr}-\frac{n-1}{r}q_r\quad\mbox{ in }(0,R)\times (T/3,T).
\end{equation}
Let $\4{q}=q+\3 r^n$ for some constant $\3>0$ to be determined later. Then
\begin{equation}
\4{q}_t-\4{q}_{rr}+\frac{n-1}{r}\4{q}_r=0\quad\mbox{ in }(0,R)\times 
(T/3,T)
\end{equation}
and 
\begin{equation}
\4{q}(0,t)=0\quad\forall T/3<t<T.
\end{equation}
By the maximum principle (\cite{F}, \cite{LSU}) there exists a constant 
$C_1>0$ such that 
$$
v_{2,r}(R,t)\le -C_1\quad\forall T/2\le t<T. 
$$
Hence by choosing $0<\3\le\3_1=C_1/(2R)$ we have
\begin{equation}
\4{q}(R,t)\le -(C_1/2)R^{n-1}\quad\forall T/2\le t<T. 
\end{equation}
By the strong maximum principle, 
\begin{align}
&v_2(r,t')<\max_{y\in\2{B}_R}v_2(y,t)=v_2(0,t)\quad\forall 0\le r\le R,
T/2\le t<t'<T\nonumber\\
\Rightarrow\quad&v_2(0,t')<v_2(0,t)\qquad\qquad\qquad\,\,\quad\forall T/2\le t<t'<T\\
\Rightarrow\quad&v_{2,t}(0,t)\le 0\qquad\qquad\qquad\qquad\,\,\quad\forall T/2\le t<T.
\nonumber
\end{align}
Suppose $v_{2,t}(0,t)\equiv 0$ on $(T/2,T)$. Then
$$
v_2(0,t')=v_2(0,t)\quad\forall T/2\le t<t'<T.
$$
This contradicts (23). Hence there exists $t_0\in (T/2,T)$ such that 
\begin{equation}
v_{2,t}(0,t_0)<0.
\end{equation}
By (18),
\begin{equation}
v_{2,t}(r,t_0)=v_{2,rr}(r,t_0)+\frac{n-1}{r}v_{2,r}(r,t_0).
\end{equation}
Since $v_{2,r}(0,t)=0$ for any $T/3\le t<T$, letting $r\to 0$ in (25) 
by the l'hosiptal rule,
\begin{equation}
v_{2,t}(0,t_0)=nv_{2,rr}(0,t_0).
\end{equation}
By (24) and (26),
$$
v_{2,rr}(0,t_0)<0.
$$
Then there exists $0<r_1<R$ such that
\begin{equation}
\3_2=-\max_{0\le r\le r_1}v_{2,rr}(r,t_0)>0.
\end{equation}
Let $\3_3=-\max_{r_1\le r\le R}v_{2,r}(r,t_0)$. Then $\3_3>0$.
Let $\3_0=\min (\3_1,\3_2/2,\3_3/(2R))$ and $0<\3\le\3_0$. Since 
$v_{2,r}(0,t_0)=0$, by the mean value theorem for any 
$0<r\le r_1$ there exists $0\le r'\le r$ such that
\begin{equation}
\4{q}(r,t_0)=r^n(\frac{v_{2,r}(r,t_0)}{r}+\3)=r^n(v_{2,rr}(r',t_0)+\3)<0.
\end{equation}
and
\begin{equation}
\4{q}(r,t_0)=r^{n-1}(v_{2,r}(r,t_0)+\3r)\le r^{n-1}(-\3_3+\3 R)<0
\quad\forall r_1\le r\le R.
\end{equation}
By (28) and (29),
\begin{equation}
\4{q}(r,t_0)<0\quad\forall 0<r\le R.
\end{equation}
By (20), (21), (22), (30) and the strong maximum principle,
\begin{align}
&\4{q}(r,t)<0\quad\forall 0<r\le R, t_0\le t<T\nonumber\\
\Rightarrow\quad&v_{2,r}(r,t)\le -\3 r\quad
\forall 0<r\le R, t_0\le t<T
\end{align}
Hence by (16) and (31),
\begin{equation*}
u_r(r,t)\le v_{2,r}(r,t)\le -\3 r\quad\forall 
0<r\le R, t_0\le t<T.
\end{equation*}
Thus
\begin{equation*}
u(r,t)=u(0,t)+\int_0^ru_r(\rho,t)\,d\rho\le 1-\3\int_0^r\rho\,d\rho
\le 1-(\3/2)r^2
\end{equation*}
holds for any $0\le r\le R$, $t_0\le t<T$, and (1) follows. By (1) 
$x=0$ is the only quenching point of $u$ and for any $n\ge 3$, 
$$
\int_{B_R}\frac{dz}{1-u(z,s)}\le C\int_0^Rr^{n-3}\,dr
\le CR^{n-2}\quad\forall t_0\le s<T
$$
and (13) follows.
\end{proof}

\begin{thm}
Let $0\le u_0\le a$ be a radially symmetric $C^2$ function 
on $\2{B}_R\subset\R^n$ for some constant $0<a<1$ which is monotone 
decreasing in $0\le r\le R$ with 
\begin{equation}
u_0\equiv 0\quad\mbox{ on }\1 B_R\quad\mbox{ and }\quad
u_0''(r)\le -c_1\quad\forall 0\le r\le R
\end{equation}
for some constant $c_1>0$. Let $\chi>0$ and $T>0$. Then there exists 
a constant $\lambda_0>0$ such that for any $\lambda\ge\lambda_0$ and 
$2<\beta<3$ if $u$ is the unique solution of ($P_{\lambda}$) in 
$Q_R^T$ given by Theorem 1, then there exists a constant $C>0$ such 
that (2) holds and for any $n\in\Z^+$ (13) holds. Moreover there 
exists a constant $\lambda_1\ge\lambda_0$ such that if $\lambda
>\lambda_1$ and $T$ is the maximal time of existence of the solution
$u$ of ($P_{\lambda}$), then $T$ satisfies (9) and $x=0$ is the only 
quenching point of $u$ at the quenching time $T$. 
\end{thm}
\begin{proof}
Note that the estimate (2) for the case $n=1$ and sufficiently 
large $\lambda$ is proved in \cite{GHW}. Here we will modify the 
proof of \cite{GHW} so that it works for any dimension $n$. 
Let $\lambda\ge\lambda_0$ and $2<\beta<3$ for some constant 
$\lambda_0>0$ to be determined later. Let $u$ be the unique 
solution of ($P_{\lambda}$) in $Q_R^T$ given by Theorem 1.
By an argument similar to the proof of Theorem 8
$0\le u<1$ is radially symmetric in $Q_R^T$ and $u$ satisfies 
(14). Let $v=1-u$ and $v_0=1-u_0$. Then $v$ satisfies
(10). By (32),
\begin{equation}
v_0=1\quad\mbox{ on }\1 B_R\quad\mbox{ and }\quad v_0'(r)\ge c_1r
\quad\forall 0\le r\le R.
\end{equation}
By (10), (33), and an argument similar to the proof of 
Proposition 2.1 of \cite{G} and Theorem 4.1 
of \cite{GHW} there exists a constant $0<c_2<c_1$ depending on
$u_0$ but is independent of $u$ such that
\begin{equation}
v_r(R,t)\ge c_2R\quad\forall 0\le t<T.
\end{equation}
Let $w=v^{\beta}$. Then $w$ satisfies
\begin{equation}
w_t-\frac{1}{r^{n-1}}(r^{n-1}w_r)_r=-\lambda\beta A(t)w^{1-\frac{3}{\beta}}
-\frac{\beta -1}{\beta}\cdot\frac{w_r^2}{w}\quad\mbox{ in }
(0,R)\times (0,T)
\end{equation}
with $w=1$ on $\1 B_R\times (0,T)$. By (14) $w_r(r,t)\ge 0$ for
any $0\le r\le R$ and $0\le t<T$.
Differentiating (35) with respect to $r$ and setting $q=r^{n-1}w_r$,
we have
\begin{align}
q_t-q_{rr}+\frac{n-1}{r}q_r=&\lambda (3-\beta)A(t)w^{-\frac{3}{\beta}}q
+\frac{\beta-1}{\beta}r^{n-1}\frac{w_r^3}{w^2}-2\frac{(\beta-1)}{\beta}
\frac{w_{rr}}{w}q\nonumber\\
\ge&\lambda (3-\beta)A(t)w^{-\frac{3}{\beta}}q
-2\frac{(\beta-1)}{\beta}\frac{w_{rr}}{w}q\quad\mbox{ in }(0,R)\times 
(0,T).
\end{align}
Now
\begin{equation}
w_{rr}q=\biggl (\frac{q}{r^{n-1}}\biggr)_rq
=\frac{q_rq}{r^{n-1}}-\frac{n-1}{r^n}q^2
\le q_rw_r.
\end{equation}
Since $0<w\le 1$, by (36) and (37),
\begin{align}
q_t-q_{rr}+\frac{n-1}{r}q_r
\ge&\frac{w_r}{w}\biggl (\lambda (3-\beta)A(t)r^{n-1}
w^{1-\frac{3}{\beta}}-2\frac{(\beta-1)}{\beta}q_r\biggr )\nonumber\\
\ge&\frac{w_r}{w}\biggl (\lambda (3-\beta)A(t)r^{n-1}
-2\frac{(\beta-1)}{\beta}q_r\biggr )
\quad\mbox{ in }(0,R)\times (0,T).
\end{align}
We now choose $\3>0$ sufficiently small such that
\begin{equation}
A(0)>\delta_1:=[1+\chi |\omega_{n-1}|(2/\3)^{\frac{1}{\beta}}
(n-(2/\beta))^{-1}R^{n-\frac{2}{\beta}}]^{-2}
\end{equation}
and
\begin{equation}
\3<\beta \min (c_1(1-a)^{\beta -1},c_2)
\end{equation}
where $\omega_{n-1}$ is the surface area of an unit ball in $\R^n$.
We next choose 
\begin{equation*}
\lambda_0>\frac{2\3 n(\beta -1)}{\beta(3-\beta)\delta_1}.
\end{equation*}
Since $\lambda\ge\lambda_0$,
\begin{equation}
\lambda (3-\beta)\delta_1>2\frac{(\beta-1)}{\beta}\3 n.
\end{equation}
By (39) and (41) there exists $0<T_1\le T$ such that $[0,T_1)$ is 
the maximal interval such that
\begin{equation}
\lambda (3-\beta)A(t)>2\frac{(\beta-1)}{\beta}\3 n
\end{equation} 
holds for any $0\le t<T_1$. Let $\4{q}=q-\3r^n$. Then by (38) 
and (42),
\begin{equation}
\4{q}_t-\4{q}_{rr}+\frac{n-1}{r}\4{q}_r
+2\frac{(\beta-1)}{\beta}\frac{w_r}{w}\4{q}_r\ge 0
\quad\mbox{ in }(0,R)\times (0,T_1).
\end{equation}
Now by (34) and (40),
\begin{equation}
\4{q}(R,t)=R^{n-1}(\beta v^{\beta-1}v_r(R,t)-\3 R)
=R^{n-1}(\beta v_r(R,t)-\3 R)\ge 0\quad\forall
0\le t<T.
\end{equation}
By (33) and (40),
\begin{equation}
\4{q}(r,0)\ge 0\quad\forall 0\le r\le R.
\end{equation}
By (43), (44), (45) and the maximum principle,
\begin{align}
&\4{q}(r,t)\ge 0\quad\forall 0\le r\le R, 0\le t<T_1\nonumber\\
\Rightarrow\quad&w_r(r,t)\ge\3 r\quad\forall 0\le r\le R, 
0\le t<T_1\nonumber\\
\Rightarrow\quad&w(r,t)\ge w(0,t)+\frac{\3}{2}r^2>
\frac{\3}{2}r^2\quad\forall 0\le r\le R, 0\le t<T_1\nonumber\\
\Rightarrow\quad&v(r,t)>\biggl (\frac{\3}{2}\biggr)^{\frac{1}{\beta}}
r^{\frac{2}{\beta}}\quad\forall 0\le r\le R, 0\le t<T_1.
\end{align}
By (46),
\begin{align}
&\int_{B_R}\frac{dz}{v(z,t)}<|\omega_{n-1}|(2/\3)^{\frac{1}{\beta}}
(n-(2/\beta))^{-1}R^{n-\frac{2}{\beta}}\quad\forall 0\le t<T_1\\
\Rightarrow\quad&A(t)>\delta_1\quad\forall 0\le t<T_1.
\end{align}
If $T_1<T$, then by (41) and (48) there exists $T_2\in (T_1,T)$
such that (42) holds on $(0,T_2)$. This contradicts the maximality
of $T_1$. Hence $T_1=T$. By (46) and (47) we get that (2) and (13) 
hold. 

We now let $\lambda>\lambda_1$ for some constant $\lambda_1$ to be 
determined later and let $T$ be the maximal time of existence of the 
solution $u$ of ($P_{\lambda}$). Suppose $T>1$. By repeating 
the above argument but with $T$ being replaced by $1$ in the 
argument and noting that then the constant $c_2$ is now 
independent of $T$ we can find constants $\3$, $\delta_1$, and 
$\lambda_0$ as before which are then all independent of $T$. Then
(2) holds in $Q_R^1$ and  (13), (48), hold for any $0\le t\le 1$.
Let 
\begin{equation}
\lambda_1=\max (\lambda_0,4n/(\delta_1R^2),3/\delta_1).
\end{equation} 
By (49) and Lemma 7, 
\begin{equation*}
1\le\frac{1}{\lambda\delta_1}
\left(1-\frac{2n}{\lambda\delta_1R^2}\right)^{-1}
\le\frac{2}{\lambda\delta_1}\le\frac{2}{3}.
\end{equation*}
Contradiction arises. Hence $T<1$. Note that then we can choose the constant 
$c_2$ in the above argument to be independent of $T$. Hence the constants 
$\3$, $\delta_1$, and $\lambda_0$ are all independent of $T$.
Then by (48) Lemma 7 holds. Hence $T$ satisfies (9). Thus 
$u$ quenches in a finite time $T$ bounded above by the right hand 
side of (9). By (2) $x=0$ is the only quenching point of $u$ at the 
quenching time $T$ and the theorem follows. 
\end{proof}


\begin{thebibliography}{99}

\bibitem{F} A.~Friedman, {\em Partial differential equations of 
parabolic type}, Prentice-Hall, Inc., Englewood Cliffs, N.J.,
U.S.A., 1964.

\bibitem{FMP} G.~Flores, G.A.~Mercado and J.A.~Pelesko, {\em Dynamics
and Touchdown in Electrostatic MEMS}, Proceedings of ICMENS (2003), 182--187.

\bibitem{GG1} N.~Ghoussoub and Y.J.~Guo, {\em On the partial differential
equations of electrostatic MEMS devices: stationary case, SIAM J.
Math. Anal.} 38 (2007), no. 5, 1423--1449.

\bibitem{GG2} N.~Ghoussoub and Y.J.~Guo, {\em On the partial 
differential equations of electrostatic MEMS devices II: dynamic case},
NoDEA Differential Equations Appl. 15 (2008), no. 1-2, 115--145.

\bibitem{GG3} N.~Ghoussoub and Y.J.~Guo, {\em On the partial 
differential equations of electrostatic MEMS devices III: refined touchdown
behavior}, J. Diff. Eqns. 244 (2008), no. 9, 2277--2309.

\bibitem{GG4} N.~Ghoussoub and Y.J.~Guo, {\em Estimates for the
quenching time of a parabolic equation modeling electrostatic MEMS}, 
Methods Appl. Anal. 15 (2008), 361--376.

\bibitem{GG5} N.~Ghoussoub and Y.J.~Guo, {\em Pull-in voltage and
steady states of nonlocal electrostatic MEMS}, preprint.

\bibitem{GHW} J.S.~Guo, B.~Hu and C.J.~Wang, {\em A nonlocal quenching 
problem arising in a micro-electro mechanical system},  Quart. Appl. 
Math. 67  (2009),  no. 4, 725--734.

\bibitem{G} Y.J.~Guo, {\em On the partial differential equations
of electrostatic MEMS devices III: Refined touchdown behavior},
J. Diff. Eqns 244 (2008), 2277--2309.

\bibitem{GPW} Y.J.~Guo, Z.G.~Pan and Michael J. Ward,
{\em Touchdown and Pull-in Voltage Behavior of a MEMS Device with
Varying Dielectric Properties}, SIAM J. Appl. Math. 66 (2005), no. 1,
309--338.

\bibitem{GW1} Z.~Guo and J.~Wei, {\em On the Cauchy problem for a 
reaction-diffusion equation with a singular nonlinearity}, J. Diff.
Eqns 240 (2007), 279--323.

\bibitem{GW2} Z.~Guo and J.~Wei, {\em On a fourth order nonlinear
elliptic equation with negative exponent},  SIAM J. Math. Anal. 40  
(2008/09),  no. 5, 2034--2054. 

\bibitem{Hs} S.Y.~Hsu, {\em Asymptotic behaviour of solutions
of the equation $u_t=\Delta\text{log }u$ near the extinction time},
Advances in Diff. Eqns. 8 (2003), no. 2, 161--187.

\bibitem{H1} K.M.~Hui, {\em Growth rate and extinction rate of a 
reaction diffusion equation with a singular nonlinearity}, Diff.
and Int. Eqns 22 (2009), nos. 7--8, 771--786.

\bibitem{H2} K.M.~Hui, {\em Global and touchdown behaviour of the generalized 
MEMS device equation}, Adv. Math. Sci. Appl. 19 (2009), no. 1, 347-370. 
 
\bibitem{H3} K.M.~Hui, {\em Existence and dynamic properties of a parabolic 
nonlocal MEMS equation}, http://arxiv.org/abs/0809.4209v2. 

\bibitem{KMS} N.I.~Kavallaris, T.~Miyasita and T.~Suzuki,{\em
Touchdown and related problems in electrostatic MEMS device equation}, 
NoDEA Nonlinear Differential Equations Appl.  15  (2008),  no. 3, 363--385. 

\bibitem{LSU} O.A.~Ladyzenskaya, V.A.~Solonnikov, and
N.N.~Uraltceva, {\em Linear and quasilinear equations of
parabolic type}, Transl. Math. Mono. Vol 23, Amer. Math. Soc., 
Providence, R.I., U.S.A., 1968.

\bibitem{LY} F.~Lin and Y.~Yang, {\em Nonlinear non-local elliptic equation
modelling electrostatic actuation}, Proc. Royal Soc. London, Ser. A 463 
(2007), 1323--1337.

\bibitem{MW} L.~Ma and J.C.~Wei, {\em Properties of postive solutions
to an elliptic equation with negative exponent}, J. Functional Analysis
254 (2008), 1058--1087.

\bibitem{P} J.A.~Pelesko, {\em Mathematical Modeling of
Electrostatic MEMS with Tailored Dielectric Properties}, SIAM J.
Appl. Math. 62 (2002), no. 3, 888--908.

\bibitem {PB} J.A.~Pelesko and D.H.~Bernstein, {\em Modeling MEMS and NEMS},
Chapman Hall and CRC Press, 2002.

\bibitem{PT} J.A.~Pelesko and A.A.~Triolo, {\em Nonlocal Problems in
MEMS Device Control}, J. Eng. Math. 41 (2001), no. 4, 345--366.

\end{thebibliography}
\end{document}